\documentclass[12pt]{amsart}

\usepackage{amsmath, amssymb, amsthm}
\usepackage{mathtools}
\usepackage{hyperref}
\usepackage{geometry}
\usepackage{abstract}
\usepackage{xcolor}
\usepackage{appendix}

\theoremstyle{plain}
\newtheorem{theorem}{Theorem}[section]
\newtheorem{proposition}[theorem]{Proposition}
\newtheorem{lemma}[theorem]{Lemma}

\theoremstyle{definition}
\newtheorem{remark}[theorem]{Remark}
\newtheorem{example}{Example}[section]

\newcommand{\R}{\mathbb{R}}

\newcommand{\eps}{\varepsilon}

\begin{document}

\title[Upper Bound for the First \texorpdfstring{$p$-Steklov}{p-Steklov} Eigenvalue in \texorpdfstring{$\mathbb{R}^n$}{R\string^n}]{Upper bound for the first $p$-Steklov eigenvalue in  $\mathbb{R}^n$}

\author{Lili Wang}
\address{Lili Wang: School of Mathematics and Statistics, Key Laboratory of Analytical Mathematics and Applications (Ministry of Education), Fujian Key Laboratory of Analytical Mathematics and Applications (FJKLAMA), Fujian Normal University, 350117 Fuzhou, China}
\email{\href{mailto:liliwang@fjnu.edu.cn}{liliwang@fjnu.edu.cn}}

\author{Tao Wang}
\address{Tao Wang: Beijing International Center for Mathematical Research, Peking University, 100871 Beijing, China}
\email{\href{mailto:taowang25@pku.edu.cn}{taowang25@pku.edu.cn}}

\date{\today}
\subjclass[2020]{35P15, 35P30, 58J50, 47J10}
\keywords{Steklov problem, $p$-Laplacian, inverse mean curvature flow, Wentzell problem}

\maketitle

\begin{abstract}
For $p\in(1,n]$ and any bounded convex domain $\Omega\subset\mathbb{R}^n$, we prove the sharp inequality
\[
\Lambda_p(\Omega):=\frac{W_p(\Omega)}{P(\Omega)V(\Omega)^{p/n}}\geq\omega_n^{-p/n},
\qquad W_p(\Omega)=\int_{\partial\Omega}|x|^p\ dS,
\]
with equality holding exactly at centered balls. Combining this with the isoperimetric inequality yields the explicit upper bound
\[
\sigma_{1,p}(\Omega)\leq \frac{A(n,p)}{r(\Omega^*)^{p-1}},
\]
where $\Omega^*$ is a ball having the same perimeter as $\Omega$, and $A(n,p)=1$ for $1<p\leq 2$, $A(n,p)=n^{p/2-1}$ for $2<p\leq n$. When $p=2$, the result recovers the higher‑dimensional Weinstock inequality of Bucur et al.\ [J.\ Differential Geom.\ 2021].  We also obtain an explicit upper bound for the first Wentzell eigenvalue of the $p$-Laplacian on convex domains.
\end{abstract}

\section{Introduction}
\label{sec:intro}
The Steklov eigenvalue problem \cite{Steklov1902} arises from the Dirichlet--to--Neumann operator and has applications in spectral geometry and hydrodynamics. For a bounded Lipschitz domain $\Omega \subset \mathbb{R}^n$ ($n \geq 2$), the first non-zero Steklov eigenvalue of the Laplacian is given by
\begin{equation}\label{eq:sigma2}
\sigma_1(\Omega) = \min_{u \in H^1(\Omega)\setminus\{0\}} \left\{\frac{\displaystyle\int_{\Omega} |\nabla u|^2\ dx}{\displaystyle\int_{\partial \Omega} u^2\ dS} \;:\; \int_{\partial \Omega} u\ dS = 0 \right\}.
\end{equation}
Weinstock \cite{Weinstock1954,Weinstock1954tech} proved that among simply connected planar domains with fixed perimeter, the disk uniquely maximizes $\sigma_1(\Omega)$. Brock \cite{Brock2001} extended this to higher dimensions under volume normalization: for any bounded Lipschitz $\Omega$,
\[
\sigma_1(\Omega) V(\Omega)^{1/n} \leq \sigma_1(B) V(B)^{1/n},
\]
with equality exactly for balls, where $V(\Omega)$ is the volume of $\Omega$. When the surface area is fixed, the ball is not a global maximizer even within the class of contractible domains, as shown by Fraser and Schoen \cite{FraserSchoen2019}. However, Bucur, Ferone, Nitsch and Trombetti \cite{BucurFeroneNitschTrombetti2021} established that within the convex class the ball remains the unique maximizer, thus recovering a sharp Weinstock-type inequality for the perimeter constraint in all dimensions. Outside the convex class the ball is in general not a maximizer \cite{FraserSchoen2019}. For non-Euclidean settings, Gu, Li and Wan \cite{GLW2025} proved a sharp Weinstock-type inequality in hyperbolic space, while Li, Wang and Wu \cite{LWW2025} established the same on Hadamard manifolds with pinched negative curvature.

The nonlinear generalization of the Laplacian to the $p$-Laplacian, defined by $\Delta_p u = \operatorname{div}(|\nabla u|^{p-2}\nabla u)$ with $p>1$, is fundamental in the study of non-Newtonian fluids and in geometric analysis. For a bounded Lipschitz domain $\Omega$, the $p$-Steklov eigenvalue problem consists in finding $\sigma \in \mathbb{R}$ and a non-trivial function $u$ such that
\begin{equation}\label{eq:p-ste}
\begin{cases}
\Delta_p u = 0 & \text{in } \Omega,\\[4pt]
|\nabla u|^{p-2}\dfrac{\partial u}{\partial \nu} = \sigma |u|^{p-2}u & \text{on } \partial\Omega,
\end{cases}
\end{equation}
where $\nu$ is the outward unit normal. The first non-zero eigenvalue $\sigma_{1,p}(\Omega)$ admits the variational characterization
\begin{equation}\label{eq:rayleigh}
\sigma_{1,p}(\Omega) = \min_{u \in W^{1,p}(\Omega)\setminus\{0\}} \left\{\frac{\displaystyle\int_{\Omega} |\nabla u|^p\ dx}{\displaystyle\int_{\partial \Omega} |u|^p\ dS} \;:\; \int_{\partial \Omega} |u|^{p-2} u \ dS = 0 \right\}.
\end{equation}
Existence, simplicity, and isolation of the eigenvalues $\sigma_{k,p}(\Omega)$ follow from standard nonlinear spectral theory \cite{Le2006,MartinezRossi2002,Torne2005}.
For general $p$, Provenzano \cite{Provenzano2022} obtained geometric upper bounds for the normalized eigenvalues $P(\Omega)^{\frac{p-1}{n-1}}\sigma_{k,p}$ in terms of the
isoperimetric ratio when $p\leq n$, where $P(\Omega)=\int_{\partial\Omega}\ dS$ is the perimeter of $\Omega$(in dimension $2$, and for surface area measure in dimension $n$), and involving an additional boundary distortion term when $p>n$. Wang and Wang \cite{WangWang2025} have recently obtained two‑sided bounds for $\sigma_{1,p}$ via isocapacitary constants. Examples in \cite{Provenzano2022} indicate that when $p>n$ the ball is unlikely to be extremal for any universal upper bound. When $p\leq n$ the bounds do not involve the boundary distortion, yet they are known to be sharp only for $p=2$ \cite{BucurFeroneNitschTrombetti2021}. Verma \cite{Verma2020} obtained upper bounds for the first non-zero eigenvalues of the closed and Steklov $p$-Laplaician eigenvalue problems on Euclidean hypersurfaces, expressed in terms of the eigenvalues of a geodesic ball of the same volume. The present work is instead based on a sharp functional inequality for the quantity $\Lambda_p$, which we introduce below, and this inequality yields an upper bound directly in terms of the perimeter. In this paper we establish an explicit upper bound for $\sigma_{1,p}(\Omega)$ valid for every $p\in(1,n]$, which reduces to the sharp Weinstock inequality when $p=2$.
\begin{theorem}\label{thm:main-bound}
Let $\Omega\subset\mathbb{R}^n$ be a bounded convex domain, and let $\Omega^*$ be a ball having the same perimeter as $\Omega$. Then for every $p\in(1,n]$,
\[
\sigma_{1,p}(\Omega)\leq 
 \begin{cases}
  \frac{1}{r(\Omega^*)^{p-1}}, \qquad &\text{if } 1< p \leq 2, \\
  n^{\frac{p}{2} - 1}\frac{1}{r(\Omega^*)^{p-1}}, &\text{if } 2< p \leq n, 
 \end{cases}
\]
where $r(\Omega^*)$ is the radius of the ball $\Omega^*$. Furthermore, equality holds if and only if $\Omega$ is a ball and $p = 2$.
\end{theorem}

\begin{remark}
The estimate in Theorem~\ref{thm:main-bound} can be rewritten as
\[
\sigma_{1,p}(\Omega) P(\Omega)^{\frac{p-1}{n-1}} \leq A(n,p) (n\omega_n)^{\frac{p-1}{n-1}},
\]
where $\omega_n$ is the Lebesgue measure of the unit ball in $\mathbb{R}^n$. When $p=2$, the right‑hand side equals the corresponding quantity for the ball, recovering the sharp Weinstock inequality. When $p\neq2$,  since the coordinate functions are no longer eigenfunctions, the first eigenvalue on the ball is strictly smaller than the upper bound. Note that for a convex domain $\Omega$, our results imply estimates in Verma \cite{Verma2020}.
\end{remark}

The proof of the main theorem relies on a sharp functional inequality for the scale invariant quantity $\Lambda_p$ that we now introduce.

Let $\Omega\subset\mathbb{R}^n$ be a bounded convex domain. Denote its volume by $V(\Omega)$, its perimeter by $P(\Omega)=\int_{\partial\Omega} dS$, where $dS$ denotes the $(n-1)$-dimensional Hausdorff measure on
$\partial\Omega$, and set
\begin{equation}\label{eq:Wp}
W_p(\Omega)=\int_{\partial\Omega}|x|^p\ dS.
\end{equation}
The relevant scale invariant functional $\Lambda_p$ is defined as
\begin{equation}\label{eq:Lambda}
\Lambda_p(\Omega)=\frac{W_p(\Omega)}{P(\Omega) V(\Omega)^{p/n}}.
\end{equation}

\begin{theorem}\label{thm:Lambda-lower}
Let $\Omega\subset\mathbb{R}^n$ be a bounded convex domain and $1<p\leq n$. Then
\[
\Lambda_p(\Omega)\geq\omega_n^{-p/n},
\]
with equality if and only if $\Omega$ is a ball centered at the origin.
\end{theorem}

\begin{remark}
For $p=2$ this is exactly the sharp inequality proved in \cite{BucurFeroneNitschTrombetti2021}.
\end{remark}

The sharp functional inequality in Theorem~\ref{thm:Lambda-lower} also yields an explicit upper bound for the first Wentzell eigenvalue. For $\beta\geq0$ the first non-zero Wentzell eigenvalue of the $p$-Laplacian is defined by
\begin{equation}\label{eq:Wentzell-def}
\mu_{1,p}(\Omega,\beta)=\inf_{v\in W^{1,p}(\Omega)\setminus\{0\}}\left\{\frac{\displaystyle\int_\Omega|\nabla v|^p dx+\beta\int_{\partial\Omega}|\nabla_\tau v|^p dS}{\displaystyle\int_{\partial\Omega}|v|^p dS}\ : \int_{\partial\Omega}|v|^{p-2}v\ dS=0\right\},
\end{equation}
where $\nabla_\tau v$ stands for the tangential gradient of $v$ on $\partial\Omega$. When $\beta=0$ the functional reduces to the $p$-Steklov quotient and $\mu_{1,p}(\Omega,0)=\sigma_{1,p}(\Omega)$, the first nontrivial $p$-Steklov eigenvalue studied above. For $p=2$ it was proved in \cite{BucurFeroneNitschTrombetti2021} that among convex domains of fixed volume the ball maximizes $\mu_{1,2}(\Omega,\beta)$, the sharp upper bound being
\begin{equation}\label{eq:Wentzell-p2}
\mu_{1,2}(\Omega,\beta)\leq\frac{1}{R}+\beta\frac{n-1}{R^2},\qquad R=\left(V(\Omega)/\omega_n\right)^{1/n},
\end{equation}
with equality exactly for balls centered at the origin.  

The same strategy, combined with Theorem~\ref{thm:Lambda-lower}, yields an explicit upper bound for every $p\in(1,n]$, with the tangential gradient term forcing a distinction between fixed volume and fixed perimeter. For fixed volume we obtain the following estimate.

\begin{theorem}\label{thm:Wentzell-V}
Let $\Omega\subset\mathbb{R}^n$ be a bounded convex domain, $p\in(1,n]$ and $\beta\geq 0$. Then
\[
\mu_{1,p}(\Omega,\beta)\leq 
 \begin{cases}
  \frac{1}{R^{p-1}}+\beta n\left(\frac{n-1}{n}\right)^{p/2}\frac{1}{R^p}, \quad &1<p \leq 2, \\
  n^{\frac{p}{2}-1}\left(\frac{1}{R^{p-1}}+\beta(n-1)\frac{1}{R^p}\right), &2<p \leq n,
 \end{cases}
\]
where $R$ is the radius of the ball centered at the origin with $V(\Omega)=\omega_n R^n$. Furthermore, equality holds if and only if $\Omega$ is a ball and $p = 2$.
\end{theorem}

\begin{remark}
For $p=2$ this reduce to the sharp estimate \eqref{eq:Wentzell-p2}. For $1<p<2$ the bound is new but not sharp: on a ball the true eigenvalue is strictly smaller because the coordinate functions are no longer eigenfunctions. As $\beta\to0^+$ the eigenvalue tends to the $p$-Steklov eigenvalue in Theorem~\ref{thm:main-bound}, while $\beta^{-1}\mu_{1,p}(\Omega,\beta)$ converges as $\beta\to+\infty$ to the first non‑zero eigenvalue of the boundary $p$-Laplace–Beltrami operator \cite{BucurFeroneNitschTrombetti2021}. The maximization of the latter under a surface area constraint is open for convex domains when $p\neq2$.
\end{remark}

\begin{remark}
Theorem~\ref{thm:main-bound} and Theorem~\ref{thm:Wentzell-V} also hold for smooth bounded star-shaped mean-convex domains. The proof is identical to that for convex domains, with the single exception that the lower bound $\Lambda_p(\Omega)\geq\omega_n^{-p/n}$ is obtained from the Kwong--Wei inequality. Indeed, For a smooth bounded star-shaped mean-convex domain, Kwong and Wei \cite[Theorem~1.1]{KwongWei2023} proved the following inequality
\[
\int_{\partial\Omega}|x|\ dS \geq \frac{n-1}{n}(n\omega_{n})^{-\frac{1}{n-1}}P(\Omega)^{\frac{n}{n-1}}+ V(\Omega).
\]
Combining the above inequality with Young's and H\"{o}lder's inequalities yields $W_p(\Omega)\geq\omega_n^{-p/n}P(\Omega)V(\Omega)^{p/n}$. The desired eigenvalue estimates then follow immediately. 
\end{remark}

The paper is organized as follows. In Section~\ref{sec:Lambda-lower} we introduce the functional $\Lambda_p$ and prove Theorem~\ref{thm:Lambda-lower}. Section~\ref{sec:preliminary} shows that after a suitable translation, the coordinate functions can serve as test functions for the $p$-Steklov problem. Section~\ref{sec:proof-main} is devoted to the proof of Theorem~\ref{thm:main-bound}. In Section~\ref{sec:Wentzell}, we adapt the analysis to Wentzell boundary conditions, proving Theorem~\ref{thm:Wentzell-V}. The detailed  derivation of the shape
derivative formula \eqref{eq:Lambda-prime} is given in Appendix~\ref{sec:appendix}.

\section{Proof of Theorem~\ref{thm:Lambda-lower}}
\label{sec:Lambda-lower}

In this section, we prove Theorem~\ref{thm:Lambda-lower}, following an approach similar to that of \cite{BucurFeroneNitschTrombetti2021}. We first recall some notation.

Let $\Omega\subset\mathbb{R}^n$ be a bounded convex domain. For any $p > 1$, consider the following scale invariant functional 
\[
 \Lambda_p(\Omega)= \frac{\displaystyle\int_{\partial\Omega}|x|^p\ dS}{P(\Omega)\ V(\Omega)^{p/n}} = \frac{W_p(\Omega)}{P(\Omega)\ V(\Omega)^{p/n}}.
\]
Let $B(r)$ be the ball of radius $r$ centered in the origin. One can check that $\Lambda_p(B(r)) = \omega_n^{-p/n}$.

\subsection{Existence of a minimizer}
\label{sec:exist}

In this subsection, we prove that for any $p>1$, $\Lambda_p(\cdot)$ admits a minimizer. We begin with the following proposition.

\begin{proposition}\label{prop:vector-barycenter}
Let $\Omega\subset\mathbb{R}^n$ be a bounded convex domain and let $p>1$. Then there exists a unique point $x_0\in\mathbb{R}^n$ such that
\begin{equation}\label{eq:vec-centroid}
\int_{\partial\Omega}|x-x_0|^{p-2}(x-x_0) dS = 0.
\end{equation}
Moreover, $x_0\in\Omega$, and the translated domain $\Omega_0 = \Omega - x_0$ satisfies $W_p(\Omega_0) \leq W_p(\Omega)$.
\end{proposition}

\begin{proof}
Define the strictly convex and coercive function
\[
F(y) = \int_{\partial\Omega} |x-y|^p \ dS, \qquad y\in\mathbb{R}^n.
\]
Since $F(y) \to +\infty$ as $|y|\to\infty$, $F$ admits a unique global minimizer $x_0\in\mathbb{R}^n$. The first-order condition $\nabla F(x_0)=0$ is exactly \eqref{eq:vec-centroid}, proving the existence and uniqueness.

We then show that $x_0\in\Omega$. Suppose $x_0 \notin \Omega$. By the Hahn--Banach theorem\cite[Theorem~1.7]{Brezis2011}, there exists a continuous linear functional $f(x) = \langle x, p \rangle$ with $p \in \mathbb{R}^n$ such that for every $x \in \Omega$, $f(x) < f(x_0)$. Since $\Omega$ is open and bounded, there exists $K \subset \partial \Omega$ with $\mathcal{H}^{n-1}(K) > 0$ such that for every $x \in K$, $f(x) < f(x_0)$. Consequently,
\[
0  = \left\langle \int_{\partial \Omega} |x - x_0|^{p-2} (x - x_0)\ dS, \  p \right\rangle \leq \int_K |x - x_0|^{p-2} \left(f(x) - f(x_0)\right) dS < 0,
\]
a contradiction, which shows that $x_0 \in \Omega$.

Finally, the inequality $W_p(\Omega_0) = F(x_0) \leq F(0) = W_p(\Omega)$ follows directly from the fact that $x_0$ minimizes $F$.
\end{proof}

As we are interested in minimal value of $\Lambda_p$, by Proposition~\ref{prop:vector-barycenter}, we can assume that $\Omega$ is in a position such that $\int_{\partial \Omega}|x|^{p-2}x dS = 0$. In particular, this implies that $0 \in \Omega$. Set $r_{\max}(\Omega)= \max\limits_{x\in\partial\Omega}|x|$ and choose a point $x_{\max}\in\partial\Omega$ with $|x_{\max}|=r_{\max}$. The excess of $\Omega$ is
\begin{equation}\label{eq:excess}
 E_p(\Omega) = r_{\max}^{p-1}(\Omega) - \frac{W_p(\Omega)}{n\ V(\Omega)}.
\end{equation}
One can check that $E_p(B(r)) = 0$.

\begin{proposition}\label{prop:exist}
 For any $p > 1$, there exists a convex set minimizing $\Lambda_p(\cdot)$.   
\end{proposition}
\begin{proof}
The proof is similar to that for the case $p = 2$. Applying Proposition~\ref{prop:vector-barycenter} to every element of a minimizing sequence $\{\Omega_j\}_{j \in \mathbb{N}}$, we obtain a new sequence $\{\widetilde\Omega_j\}_{j\in\mathbb{N}}$ with the additional property that $0\in\widetilde\Omega_j$ for all $j$. Moreover, by translation invariance of $P$ and $V$ and the inequality $W_p(\widetilde\Omega_j)\leq W_p(\Omega_j)$, we have $\Lambda_p(\widetilde\Omega_j)\leq\Lambda_p(\Omega_j)$. Hence the new sequence is also a minimizing sequence.

After a harmless scaling, we may also assume that $V(\widetilde\Omega_j)=\omega_n$, as $\Lambda_p(\cdot)$ is scale invariant. We claim that the diameters of $\widetilde\Omega_j$ are uniformly bounded. Indeed, if a subsequence had $\operatorname{diam}(\widetilde\Omega_j)\to\infty$, convexity and the fixed volume would force $P(\widetilde\Omega_j)\to\infty$ (see, e.g., \cite[Lemma~4.1]{EspositoFuscoTrombetti2005} or \cite{BettaBrockMercaldoPosteraro1999}). At the same time, since each $\widetilde\Omega_j$ contains the origin, a fixed fraction of its boundary lies at distance at least $\frac12\operatorname{diam}(\widetilde\Omega_j)$ from the origin. Consequently, there exists some constant $c=c(n,p)>0$, independent of $j$, such that $W_p(\widetilde\Omega_j)/P(\widetilde\Omega_j)\geq c\ \left(\operatorname{diam}(\widetilde\Omega_j)\right)^p$. Combining this with $P(\widetilde\Omega_j)\to\infty$ forces $\Lambda_p(\widetilde\Omega_j)\to\infty$, contradicting the fact that the sequence is minimizing. Hence the diameters are uniformly bounded.

The Blaschke selection theorem \cite[Theorem 1.8.6]{Rolf2014} then yields a subsequence converging in the Hausdorff distance to a compact convex set $\Omega$, which by continuity of $V$, $P$ and $W_p$ is a minimizer of $\Lambda_p(\cdot)$.
\end{proof}

\subsection{A minimizer cannot have negative excess}
\label{sec:neg}

In this subsection, we show that a minimizer of $\Lambda_p(\cdot)$ cannot have negative excess. We first observe that there exist convex sets with negative excess.
\begin{example}
Let
\[
\mathcal{E}_{\varepsilon}:=\left\{(x,y)\in\mathbb{R}^2:\frac{x^2}{\varepsilon^2}+\varepsilon^2y^2=1\right\},
\]
where $\varepsilon > 0$ is small. Then $r_{\max}(\mathcal{E}_{\varepsilon})=1/\varepsilon$, $V(\mathcal{E}_{\varepsilon})=\pi$, and
\[
E_p(\mathcal{E}_{\varepsilon})=\frac{1}{\varepsilon^{p-1}}-\frac{W_p(\mathcal{E}_{\varepsilon})}{2\pi}.
\]
Parameterizing $\partial\mathcal{E}_{\varepsilon}$ by $x=\varepsilon\cos\theta$, $y=\varepsilon^{-1}\sin\theta$ ($\theta\in[0,2\pi)$), we have
\[
|x|^p = (\varepsilon^2\cos^2\theta+\varepsilon^{-2}\sin^2\theta)^{p/2},\qquad dS = \sqrt{\varepsilon^2\sin^2\theta+\varepsilon^{-2}\cos^2\theta}\ d\theta.
\]
Thus,
\begin{align*}
W_p(\mathcal{E}_{\varepsilon}) &=\int_0^{2\pi}(\varepsilon^2\cos^2\theta+\varepsilon^{-2}\sin^2\theta)^{p/2}\sqrt{\varepsilon^2\sin^2\theta+\varepsilon^{-2}\cos^2\theta}\ d\theta\\
&\geq \int_0^{2\pi}\frac{1}{\varepsilon^p} |\sin\theta|^p\frac{1}{\varepsilon}|\cos \theta|\ d\theta\\
&=\frac{4\int_0^\frac{\pi}{2}\sin^p\theta\cos\theta\ d\theta}{\varepsilon^{p+1}}=\frac{4}{(p+1)\varepsilon^{p+1}}.
\end{align*}
Hence,
\begin{equation*}
E_p(\mathcal{E}_{\varepsilon}) \leq\frac{1}{\varepsilon^{p-1}} \left(1-\frac{2}{(p+1)\varepsilon^2\pi}\right)<0  
\end{equation*}
for $\varepsilon$ sufficiently small.
\end{example}

When $\Omega$ is a bounded open convex set with $C^2$ boundary, we denote by $\nu(x)$ the outward unit normal at $x\in\partial\Omega$ and by $H_{\partial\Omega}(x)$ the mean curvature, defined as the average of the principal curvatures of $\partial\Omega$ with respect to $-\nu$. 

Let $\Omega$ be a bounded open convex set with $C^2$ boundary. For a smooth vector field $\theta\in C^\infty(\mathbb{R}^n,\mathbb{R}^n)$, set $\Phi_s(x)=x+s\theta(x)$ and $\Omega_s=\Phi_s(\Omega)$ for small $|s|$. The directional derivative of $\Lambda_p$ at $\Omega$ in the direction $\theta$ is defined by the limit
\[
\Lambda_p'(\Omega)[\theta] = \frac{d}{ds}\bigg|_{s=0}\Lambda_p(\Omega_s) = \lim_{s\to 0}\frac{\Lambda_p(\Omega_s)-\Lambda_p(\Omega)}{s}.
\]
Applying the classical Hadamard shape derivative formulas to the volume and boundary integrals that appear in $\Lambda_p$, together with the quotient rule, one obtains (see, e.g., \cite[Theorem~5.2.2]{HenrotPierre2005}, \cite[Section~2.11]{SZ92} for the differentiation of volume and boundary integrals)
\begin{equation}\label{eq:Lambda-prime}
\begin{split}
\Lambda_p'(\Omega)[\theta]
= \frac{1}{P(\Omega) V(\Omega)^{p/n}} \int_{\partial\Omega}
   &\left[p\varphi\left(|x|^{p-2}\langle x,\nu\rangle - \frac{W_p(\Omega)}{n V(\Omega)}\right)\right. \\
  &+\left.(n-1)H_{\partial\Omega} \varphi\left(|x|^p - \frac{W_p(\Omega)}{P(\Omega)}\right)\right]\ dS,
\end{split}
\end{equation}
where $\varphi(x)=\langle\theta(x),\nu(x)\rangle$ for $x\in\partial\Omega$. We give a detailed derivation of equation \eqref{eq:Lambda-prime} in the appendix.

Let $\mathcal{K}^{\infty,+}$ denote the class of compact convex subsets of $\mathbb{R}^n$ with nonempty interior, $C^\infty$ boundary, and strictly
positive principal curvatures. For $\Omega\in\mathcal{K}^{\infty,+}$, the inverse mean curvature flow (IMCF) produces a smooth family $\{\Omega_t\}_{t\geq 0}\subset\mathcal{K}^{\infty,+}$ whose boundary moves with normal velocity $1/H_{\partial\Omega_t}$. This flow exists for all $t\geq0$ and preserves the class $\mathcal{K}^{\infty,+}$
\cite{Gerhardt1990,Urbas1990}. Substituting $\theta = \frac{1}{H_{\partial\Omega}}\nu$ into \eqref{eq:Lambda-prime} and noting that the second integral vanishes since
$\int_{\partial\Omega}\left(|x|^p - \frac{W_p(\Omega)}{P(\Omega)}\right)\ dS = 0$, we obtain
\begin{equation}\label{eq:Lambda-IMCF}
\frac{d}{dt}\bigg|_{t=0} \Lambda_p(\Omega_t)
= \frac{p}{P(\Omega) V(\Omega)^{p/n}}\int_{\partial\Omega} \frac{1}{H_{\partial\Omega}}
   \left(|x|^{p-2}\langle x,\nu\rangle - \frac{W_p(\Omega)}{n V(\Omega)}\right)\ dS.
\end{equation}
Since $|x|^{p-2}\langle x,\nu\rangle \leq |x|^{p-1} \leq r_{\max}^{p-1}$ pointwise on $\partial\Omega$, the integrand is bounded above by $E_p(\Omega)$. The sign of the derivative is therefore governed by the excess.

\begin{proposition}\label{prop:negative-excess}
Let $\Omega\subset\mathbb{R}^n$ be a bounded convex domain and $1<p\leq n$. If $E_p(\Omega)<0$, then $\Omega$ cannot be a minimizer of $\Lambda_p$ among convex domains.
\end{proposition}

\begin{proof}
We argue by contradiction, following the approach of Bucur et al.~\cite{BucurFeroneNitschTrombetti2021}. Assume that $\Omega$ is a minimizer with $E_p(\Omega)<0$.

Suppose first that $\Omega\in\mathcal{K}^{\infty,+}$.  Under the IMCF, the right‑hand side of \eqref{eq:Lambda-IMCF} is strictly negative since the integrand is pointwise bounded above by $E_p(\Omega)<0$ and $\frac{1}{H_{\partial\Omega}}>0$. This contradicts the minimality of $\Omega$.

If $\Omega\notin\mathcal{K}^{\infty,+}$, take a sequence $\{\Omega_k\}_k\subset\mathcal{K}^{\infty,+}$ with $\Omega_{k+1}\subset\Omega_k$ and $\Omega_k\to\Omega$ in the Hausdorff
distance. On the class of convex sets, $V$, $P$, $W_p$ and $r_{\max}$ are continuous with respect to Hausdorff convergence \cite[Lemma~2.1]{BucurFeroneNitschTrombetti2021}. Hence
\begin{equation}\label{eq:conv}
\begin{aligned}
\lim_{k\to\infty}V(\Omega_k)&=V(\Omega), &
\lim_{k\to\infty}P(\Omega_k)&=P(\Omega),\\
\lim_{k\to\infty}W_p(\Omega_k)&=W_p(\Omega), &
\lim_{k\to\infty}r_{\max}(\Omega_k)&=r_{\max}(\Omega).
\end{aligned}
\end{equation}

For each $k$, let $\Omega_k(t)$ be the IMCF starting from $\Omega_k$, so that $\Omega_k(0)=\Omega_k$ and the flow exists for all $t\geq 0$. The classical evolution laws
\begin{align}
\frac{d}{dt}V(\Omega_k(t)) &= \int_{\partial\Omega_k(t)}\frac{1}{H_{\partial\Omega_k(t)}}\ dS,
\label{eq:Vdot}\\[2pt]
\frac{d}{dt}P(\Omega_k(t)) &= (n-1)P(\Omega_k(t)),
\label{eq:Pdot}\\[2pt]
\frac{d}{dt}r_{\max}(\Omega_k(t)) &\leq r_{\max}(\Omega_k(t))
\label{eq:rmax}
\end{align}
hold for all $t\geq0$ \cite{Gerhardt1990,Urbas1990}. From \eqref{eq:rmax} we obtain the integrated estimate
\begin{equation}\label{eq:rmax-exp}
r_{\max}(\Omega_k(t)) \leq r_{\max}(\Omega_k)e^{t}.
\end{equation}

We now apply the shape derivative formula to the IMCF at time $t$. For the IMCF the velocity field is $\theta = \frac{1}{H_{\partial\Omega_k(t)}}\nu$, and the second integral in \eqref{eq:Lambda-prime} vanishes. Set $V_k(t)=V(\Omega_k(t))$, $P_k(t)=P(\Omega_k(t))$ and $W_k(t)=W_p(\Omega_k(t))$. We obtain
\begin{equation}\label{eq:Lambda-prime-t}
\frac{d}{dt}\Lambda_p(\Omega_k(t))
= \frac{p}{P_k(t) V_k(t)^{p/n}}
   \int_{\partial\Omega_k(t)} \frac{1}{H_{\partial\Omega_k(t)}}
   \left(|x|^{p-2}\langle x,\nu\rangle - \frac{W_k(t)}{n V_k(t)}\right)\ dS.
\end{equation}
The pointwise bound $|x|^{p-2}\langle x,\nu\rangle \leq r_{\max}^{p-1}(\Omega_k(t))$ together with \eqref{eq:rmax-exp} gives
\[
|x|^{p-2}\langle x,\nu\rangle \leq r_{\max}^{p-1}(\Omega_k) e^{(p-1)t}.
\]

Since $\Omega$ is a minimizer, $\Lambda_p(\Omega) \leq \Lambda_p(\Omega_k(t))$. By the definition of $\Lambda_p$, this is equivalent to
\[
\frac{W_k(t)}{P_k(t)} \geq \Lambda_p(\Omega) V_k(t)^{p/n}.
\]
The perimeter is nondecreasing along the IMCF by \eqref{eq:Pdot}, and the nested inclusion of the approximating sequence guarantees $P_k(t) \geq P_k(0) = P(\Omega_k) \geq P(\Omega)$. Therefore
\[
W_k(t) \geq \Lambda_p(\Omega) V_k(t)^{p/n} P(\Omega).
\]
Inserting these estimates into \eqref{eq:Lambda-IMCF} and using $V_k'(t) = \frac{d}{dt}V(\Omega_k(t))$, we arrive at
\begin{equation}\label{eq:Ldot}
\frac{d}{dt}\Lambda_p(\Omega_k(t))
\leq \frac{p\ V_k'(t)}{P(\Omega_k) V(\Omega_k)^{p/n}}
   \left(r_{\max}^{p-1}(\Omega_k) e^{(p-1)t}
        - \frac{\Lambda_p(\Omega) P(\Omega)}{n V_k(t)^{1-p/n}}\right).
\end{equation}

We now turn to the key differential inequality that leads to a contradiction. From \eqref{eq:Ldot} and the monotonicity of the volume along the flow we obtain, for any $0<t<T$, since $1<p\leq n$,
\[
r_{\max}^{p-1}(\Omega_k) e^{(p-1)t}
- \frac{\Lambda_p(\Omega)P(\Omega)}{n V_k(t)^{1-p/n}}
\leq r_{\max}^{p-1}(\Omega_k) e^{(p-1)T}
- \frac{\Lambda_p(\Omega)P(\Omega)}{n V_k(T)^{1-p/n}} .
\]
Integrating \eqref{eq:Ldot} from $0$ to $T$ and using the above monotonicity gives
\begin{equation}\label{eq:Lambda-diff}
\begin{aligned}
\Lambda_p(\Omega_k(T))-\Lambda_p(\Omega_k)
\leq \frac{p\left(V_k(T)-V_k(0)\right)}{P(\Omega_k) V(\Omega_k)^{p/n}}
&\Bigl(r_{\max}^{p-1}(\Omega_k) e^{(p-1)T} \\
      &- \frac{\Lambda_p(\Omega)P(\Omega)}{n V_k(T)^{1-p/n}}\Bigr).
\end{aligned}
\end{equation}

Since $E_p(\Omega)<0$, the continuous function
\[
\phi(v)=r_{\max}^{p-1}(\Omega)-\frac{W_p(\Omega)}{n v^{1-p/n}}
\]
is strictly negative at $v=V(\Omega)$. By continuity, $\phi(v)<0$ for all $v$ sufficiently close to $V(\Omega)$. Together with the convergences $V(\Omega_k)\to V(\Omega)$ and
$r_{\max}(\Omega_k)\to r_{\max}(\Omega)$, we can pick $T>0$ small enough and an integer $k_0$ such that
\[
V_k(T) < V(\Omega)+\delta,\qquad
r_{\max}^{p-1}(\Omega_k) e^{(p-1)T} - \frac{\Lambda_p(\Omega)P(\Omega)}{n(V(\Omega) + \delta)^{1-p/n}} < 0
\]
for all $k\geq k_0$. Applying Ros' inequality \cite[Theorem 1]{Ros1987} in $\mathbb{R}^n$, we obtain $V_k'(t)\geq n V_k(t)$, which implies
\[
V_k(T)-V_k(0) \geq V(\Omega_k)(e^{nT}-1) > 0.
\]

Substituting these estimates into \eqref{eq:Lambda-diff} and passing to the limit yields
\[
\lim_{k \to \infty}\Lambda_p(\Omega_k(T))-\Lambda_p(\Omega) < 0.
\]
The sequence $\{\Omega_k(T)\}$ is uniformly bounded, so a subsequence converges in the Hausdorff distance to a convex set $\widetilde\Omega$. By the continuity of $\Lambda_p$, we obtain $\Lambda_p(\widetilde\Omega) < \Lambda_p(\Omega)$, contradicting the minimality of $\Omega$.
\end{proof}

\subsection{A minimizer cannot have positive excess}
\label{sec:pos}
In this subsection we show that a bounded convex domain $\Omega$ with $E_p(\Omega) > 0$ cannot be a minimizer of $\Lambda_p(\cdot)$. We first observe that there exists sets having positive excess.
\begin{example}
Let $p>1$ and consider the ellipse
\[
\mathcal{E}_\varepsilon = \left\{(x,y)\in\mathbb{R}^2 : \frac{x^2}{(1+\varepsilon)^2}+\frac{y^2}{(1-\varepsilon)^2}=1 \right\},
\qquad 0<\varepsilon\ll 1.
\]
Using the standard parametrization $x=(1+\varepsilon)\cos\theta$ and $y=(1-\varepsilon)\sin\theta$, we have
\[
V(\mathcal{E}_\varepsilon)= \pi(1-\varepsilon^2),\qquad r_{\max}(\mathcal{E}_\varepsilon)=1+\varepsilon,
\]
\[
x^2 + y^2 = 1+\varepsilon^2+2\varepsilon\cos 2\theta,\qquad
dS = \left(1-\varepsilon\cos 2\theta + O(\varepsilon^2)\right)d\theta.
\]
A Taylor expansion yields
\begin{align*}
W_p(\mathcal{E}_{\varepsilon}) &= \int_{\partial \mathcal{E}_\varepsilon} (x^2 + y^2)^{\frac{p}{2}}\ dS \\
&= \int_0^{2\pi} \left(1 + p\varepsilon\cos 2\theta + O(\varepsilon^2)\right)          \left(1 - \varepsilon\cos 2\theta + O(\varepsilon^2)\right) d\theta= 2\pi +O(\varepsilon^2).
\end{align*}
Consequently,
\[
\frac{W_p(\mathcal{E}_{\varepsilon})}{2V(\mathcal{E}_{\varepsilon})} = \frac{2\pi+O(\varepsilon^2)}{2\pi(1-\varepsilon^2)} = 1+O(\varepsilon^2),
\]
while
\[
r_{\max}^{p-1} = (1+\varepsilon)^{p-1} = 1+(p-1)\varepsilon + O(\varepsilon^2).
\]
The excess therefore expands as
\[
E_p(\mathcal{E}_\varepsilon) = r_{\max}^{p-1} - \frac{W_p(\mathcal{E}_{\varepsilon})}{2V(\mathcal{E}_{\varepsilon})} = (p-1)\varepsilon + O(\varepsilon^2).
\]
For $p>1$, the linear term dominates, and $E_p(\mathcal{E}_\varepsilon) > 0$ for sufficiently small $\varepsilon>0$.    
\end{example}

For a bounded convex domain $\Omega$, we show that when $E_p(\Omega)>0$, removing a small cap near a point where $|x|=r_{\max}$ strictly decreases $\Lambda_p$. By Proposition~\ref{prop:vector-barycenter} we may assume that $\Omega$ contains the origin. For a small $\varepsilon>0$, let $\nu_{\max}=\frac{x_{\max}}{r_{\max}}$, we define the half-space
\[
H_{\varepsilon} = \{x\in\mathbb{R}^n : \langle x,\nu_{\max}\rangle \leq r_{\max}-\varepsilon\},
\]
and the truncated convex set $\Omega_{\varepsilon} = \Omega \cap H_{\varepsilon}$ and $U_{\varepsilon}=\Omega_{\varepsilon}\cap \partial H_{\varepsilon}$. Set 
\begin{equation}\label{v-P-W-differ}
\Delta V= V(\Omega_\varepsilon)-V(\Omega),\ \  \Delta P=P(\Omega_\varepsilon)-P(\Omega), \ \ 
\Delta W_p  = W_p(\Omega_\varepsilon)-W_p(\Omega).    
\end{equation}
The quantities $\Delta W_p$, $\Delta V$, $\Delta P$ depend on the truncation parameter $\varepsilon$ and vanish as $\varepsilon\to 0^+$.  

\begin{lemma}\label{lem:expansions}
There exists a constant $C(\Omega)>0$ such that for all sufficiently small
$\varepsilon>0$,
\begin{enumerate}
  \item[(i)] $|\Delta V| \leq C(\Omega) |\Delta P|$;
  \item[(ii)] $ \Delta W_p = p r_{\max}^{p-1}\Delta V+ r_{\max}^{p} \Delta P + o(\Delta V)+o(\Delta P)$ as $\varepsilon\to0^+$.
\end{enumerate}
\end{lemma}

\begin{proof}
The strategy follows that of Bucur et al. \cite[Lemma 2.5]{BucurFeroneNitschTrombetti2021}. We provide the complete details for the sake of completeness. The argument uses only the convexity of $\Omega$ and $\Omega_{\varepsilon}$, no regularity of its boundary is required.

Since $\Omega$ is a convex body contained in the ball $B(r_{\max})$ and $U_{\varepsilon}\subset B(r_{\max})\cap \partial H_{\varepsilon}$, we have 
\begin{equation}\label{eq:diam-est}
\frac{\operatorname{diam}(U_{\varepsilon})}{2}\leq \sqrt{r^2_{\max}(\Omega)-\left(r_{\max}(\Omega)-\varepsilon\right)^2}\leq \sqrt{2r_{\max}(\Omega)\varepsilon}.    
\end{equation}
We rotate the coordinate system so that the positive $x_n$-axis points $\nu_{\max}$. Let $U_{\varepsilon}' \subset \mathbb R^{n-1}$ be the orthogonal projection onto the subspace $\{x_n=0\}$ of the upper graph of $\Omega_{\varepsilon}$. Equivalently, there exists a function $f: U_{\varepsilon}' \to [r_{\max} -\varepsilon, r_{\max}]$ such that the upper boundary $\partial\Omega\setminus \partial\Omega_{\varepsilon}=\{(y, f(y)) : y \in U_{\varepsilon}'\}$. By construction, $f$ is concave and attains its maximum at the origin, i.e.,
$f(0) = \max_{y\in U_{\varepsilon}'} f(y) =r_{\max}$. Set $\phi(y)=f(y)-(r_{\max}-\varepsilon)$ on $U_{\varepsilon}'$. Since $\phi$ is concave and nonnegative, with maximum $\phi(0)=\varepsilon$, the graph of $\phi$ lies above the cone with vertex $x_{\max}$ and base $U_{\varepsilon}$. Consequently, using the classical cone-volume estimate yields 
\begin{equation}\label{eq:V-Excess}
-\Delta V=|\Delta V|=\int_{U_{\varepsilon}'} \phi(y)\ dy \geq \frac{\varepsilon}{n}\mathcal{H}^{n-1}(U_{\varepsilon}').    
\end{equation}
Moreover, by Sobolev Poincar\'{e} inequality, we have
\begin{equation}\label{ineq: V-excess-square}
\int_{U_{\varepsilon}'}\phi^2\ dy\leq C(n)(\mathcal{H}^{n-1}(U_{\varepsilon}'))^{\frac{2}{n-1}}\int_{U_{\varepsilon}'}|D \phi|^2\ dy,
\end{equation}
here $D$ denotes the gradient w.r.t. the $y$-variables in $\mathbb{R}^{n-1}$. Combining \eqref{eq:diam-est}, \eqref{eq:V-Excess} and \eqref{ineq: V-excess-square}, we obtain
\begin{equation}\label{eq:Delta-V-uppbd}
\begin{aligned}
|\Delta V|
&\leq |\Delta V|^2 \frac{n}{\varepsilon\mathcal{H}^{n-1}(U_{\varepsilon}')}= \left(\int_{U_{\varepsilon}'}\phi\ dy\right)^2\frac{n}{\varepsilon\mathcal{H}^{n-1}(U_{\varepsilon}')}\\
&\leq C(n)\frac{(\mathcal{H}^{n-1}(U_{\varepsilon}'))^\frac{2}{n-1}}{\varepsilon}\int_{U_{\varepsilon}'}|D\phi|^2\ dy\\
&\leq C(n) 2r_{\max}(\Omega)(\omega_{n-1})^\frac{2}{n-1}\int_{U_{\varepsilon}'}|D\phi|^2\ dy 
\end{aligned}    
\end{equation}

Fix $\varepsilon_1 = \varepsilon_1(\Omega)$ such that $\partial \Omega \setminus \partial \Omega_{\varepsilon_1}$ can be represented by the graph of $f$. Set $M(\Omega) = \|Df\|_{L^\infty(U_{\varepsilon_1}')}$. Note that $M(\Omega)$ is finite since $f$ is Lipschitz on $U_{\varepsilon_1}'$. Then for any $\varepsilon < \varepsilon_1$, we have
\begin{equation}\label{eq:Delta P}
-\Delta P=|\Delta P| = \int_{U_{\varepsilon}'} \left(\sqrt{1+|Df|^2} - 1 \right) dy \geq K(\Omega)\int_{U_{\varepsilon}'} |Df|^2\ dy,    
\end{equation}
where we used 
\[
\sqrt{1+|Df|^2} - 1\geq \frac{|Df|^2}{2\sqrt{1+|Df|^2}}\geq \frac{|Df|^2}{2+2M(\Omega)}.
\]
Collecting \eqref{eq:Delta-V-uppbd} and \eqref{eq:Delta P} gives the first conclusion (i).  
 
We turn to the estimate for $\Delta W_p$. Set $R:=r_{\max}$ and define
\[
a(y)=(R-\varepsilon)^2+|y|^2,\qquad b(y)=f^2(y)+|y|^2=\left[R-\varepsilon+\phi(y)\right]^2+|y|^2.
\]
Then
\begin{equation}\label{eq:int-W_p}
\begin{aligned}
-\Delta W_p
&=\int_{U_{\varepsilon}'}\left(f^2(y)+|y|^2\right)^\frac{p}{2}\sqrt{1+|Df|^2}\ dy-\int_{U_{\varepsilon}'}a^\frac{p}{2}(y)\ dy\\
&=\int_{U_{\varepsilon}'}(b^{p/2}(y)-a^{p/2}(y))\ dy+\int_{U_{\varepsilon}'} b^{p/2}(y)\left(\sqrt{1+|D\phi|^2}-1\right) dy\\
\end{aligned}
\end{equation}
Note that $b(y)-a(y)= \phi(y)(2R - 2\varepsilon + \phi(y))$. By the mean value theorem, we have $b^{p/2}(y)-a^{p/2}(y)=\frac{p}{2} \xi^{p/2-1}(y)(b(y)-a(y))$ for some $\xi(y)$ between $a(y)$ and $b(y)$. As $|y|^2\leq \frac{1}{4}(\operatorname{diam}(U_{\varepsilon}))^2=O(\varepsilon)$ and $0\leq \phi\leq\varepsilon$, both $a(y),b(y)=R^2+O(\varepsilon)$, whence $\xi(y)=R^2+O(\varepsilon)$ uniformly on $U_{\varepsilon}'$. Then
\begin{align*}
b^{p/2(y)}-a^{p/2}(y) 
&=\frac{p}{2}(R^2+O(\varepsilon))^{p/2-1}\phi(y)(2R+O(\varepsilon))\\
&=\frac{p}{2}(R^{p-2}+O(\varepsilon))\phi(y)(2R+O(\varepsilon)) =pR^{p-1}\phi(y)+\phi(y)O(\varepsilon).
\end{align*}
The first integral in \eqref{eq:int-W_p} equals
\begin{equation}\label{eq:integral-term1}
\begin{split}
\int_{U_{\varepsilon}'}(b^{p/2}(y)-a^{p/2}(y))\ dy &= pR^{p-1}\int_{U_{\varepsilon}'}\phi(y)\ dy + o(|\Delta V|) \\
&=-pR^{p-1}\Delta V-o(\Delta V),   
\end{split}
\end{equation}
where $-\Delta V=\int_{U_{\varepsilon}'}\phi(y)\ dy$. Similarly, due to the facts that $b^{p/2}(y)=R^p+O(\varepsilon)$ and $-\Delta P=\int_{U_{\varepsilon}'}(\sqrt{1+|D \phi|^2}-1)\ dy$, the second integral in \eqref{eq:int-W_p} becomes
\[
\left(R^p+O(\varepsilon)\right)\int_{U_{\varepsilon}'}\left(\sqrt{1+|\nabla h|^2}-1\right) dy=-R^p\Delta P-o(\Delta P).
\]
Therefore, 
\[
-\Delta W_p= -pR^{p-1}\Delta V - R^p\Delta P - o(\Delta V)+o(\Delta P).
\]
Combining this with the earlier lower bounds for $\Delta V$ and $\Delta P$ gives the second conclusion.
\end{proof}

\begin{proposition}\label{prop:positive}
Let $\Omega$ be a bounded convex domain of $\mathbb{R}^n$ such that
\[
E_p(\Omega)>0,
\]
then $\Omega$ is not a minimizer of $\Lambda_p(\cdot)$.
\end{proposition}

\begin{proof}
We argue by contradiction. Assume that $\Omega$ is a minimizer with $E_p(\Omega)>0$. By Proposition~\ref{prop:vector-barycenter} we may assume that $0\in\Omega$.

Recall the notation introduced before Lemma~\ref{lem:expansions}. For brevity we write
\[
V = V(\Omega),\quad P = P(\Omega),\quad W_p = W_p(\Omega), \quad\Lambda_p=\Lambda_p(\Omega).
\]
From the definition of $\Lambda_p$ we have
\[
\Lambda_p(\Omega_\varepsilon) = \frac{W_p + \Delta W_p}{(V + \Delta V)^{p/n} (P + \Delta P)},
\qquad
\Lambda_p= \frac{W_p}{V^{p/n} P}.
\]
A first-order expansion in the small ratios $\Delta V/V$, $\Delta P/P$, and $\Delta W_p/W_p$, with higher-order terms collected in the remainder, gives
\[
\begin{aligned}
\frac{\Delta\Lambda_p}{\Lambda_p}
&=\frac{\Lambda_p(\Omega_\eps)-\Lambda_p}{\Lambda_p}
= \frac{W_p+\Delta W_p}{W_p}\left(\frac{V}{V+\Delta V}\right)^{p/n}\frac{P}{P+\Delta P} - 1\\
&= \left(1+\frac{\Delta W_p}{W_p}\right)
   \left(1-\frac{p}{n}\frac{\Delta V}{V}+O(|\Delta V|^2)\right)
   \left(1-\frac{\Delta P}{P}+O(|\Delta P|^2)\right) - 1\\
&= \frac{\Delta W_p}{W_p} - \frac{p}{n}\frac{\Delta V}{V} - \frac{\Delta P}{P}
   + O\left(|\Delta V|^2+|\Delta P|^2+|\Delta W_p|^2\right).
\end{aligned}
\]
By Lemma~\ref{lem:expansions},
\[
\frac{\Delta\Lambda_p}{\Lambda_p}= \Delta V\left(\frac{p r_{\max}^{p-1}}{W_p} - \frac{p}{nV}\right)
+ \Delta P\left(\frac{r_{\max}^{p}}{W_p} - \frac{1}{P}\right) + o(|\Delta V|+|\Delta P|).
\]
The assumption $E_p(\Omega)>0$ gives
\[
A := \frac{p r_{\max}^{p-1}}{W_p} - \frac{p}{nV} = \frac{p}{W_p}E_p(\Omega) > 0.
\]
Since $|x|\leq r_{\max}$ on $\partial\Omega$, we have $W_p \leq r_{\max}^p P$. Hence
\[
B := \frac{r_{\max}^{p}}{W_p} - \frac{1}{P} \geq 0.
\]
Since $\Delta P<0$ and $\Delta V<0$,
\[
\frac{\Delta\Lambda_p}{\Lambda_p} \leq A\Delta V + o(|\Delta V|+|\Delta P|).
\]
For small $\varepsilon$ the $o(\cdot)$ term is smaller than $\frac12|A\Delta V|$, so the right‑hand side is negative. Hence $\Lambda_p(\Omega_\varepsilon)<\Lambda_p(\Omega)$, which contradicts the minimality of $\Omega$. Thus $E_p(\Omega)>0$ is impossible, proving the proposition.
\end{proof}

\subsection{A minimizer with zero excess is a centered ball}
\begin{proposition}\label{prop:zero-excess}
Let $\Omega\subset\R^n$ be a bounded convex domain. Assume that $\Omega$ is a minimizer of $\Lambda_p$ with $E_p(\Omega)=0$. Then $\Omega$ is a ball centered at the origin.
\end{proposition}

\begin{proof}
By Proposition~\ref{prop:vector-barycenter} we may translate $\Omega$ so that $0\in\Omega$ and $\Lambda_p$ does not increase; hence we assume $0\in\Omega$ without loss of generality.  Suppose that $\Omega$ is not a centered ball. Let $\Omega_\varepsilon$ be the truncation defined before Lemma~\ref{lem:expansions} and keep the notation $\Delta V,\Delta P,\Delta W_p$ from there. By construction, $\Delta V<0$ and $\Delta P<0$.

From Lemma~\ref{lem:expansions} and the computation of the relative change carried out in Proposition~\ref{prop:positive}, we have
\begin{equation}\label{E_p=0-differ}
\frac{\Delta\Lambda_p}{\Lambda_p}
= \Bigl(\frac{r_{\max}^p}{W_p} - \frac{1}{P}\Bigr)\Delta P
  + o(|\Delta V|+|\Delta P|),
\end{equation}
where we already used $E_p(\Omega)=0$ to cancel the $\Delta V$ term.

Since $|x|\leq r_{\max}$ on $\partial\Omega$, we have $W_p\leq r_{\max}^p P$. Equality would imply $|x|=r_{\max}$ $\mathcal{H}^{n-1}$-a.e., which together with convexity and $0\in\Omega$ forces $\Omega = B(r_{\max})$, which contradicts our assumption. Hence $W_p < r_{\max}^p P$, so the coefficient
\[
C := \frac{r_{\max}^p}{W_p} - \frac{1}{P} > 0.
\]

Since $\Delta P<0$, the principal term $C\Delta P$ is strictly negative.  Choosing $\varepsilon$ small enough makes the remainder $o(|\Delta V|+|\Delta P|)$ negligible compared to $|\Delta V|+|\Delta P|$, so the right‑hand side of \eqref{E_p=0-differ} becomes strictly negative. Thus $\Lambda_p(\Omega_\varepsilon) < \Lambda_p(\Omega)$, contradicting the minimality of $\Omega$. Therefore, $\Omega$ must be a centered ball.
\end{proof}

\begin{proof}[Proof of Theorem~\ref{thm:Lambda-lower}]
From Subsection~\ref{sec:exist}, it follows that $\Lambda_p(\cdot)$ admits a minimizer among all bounded convex domains that contain the origin. According to Propositions~\ref{prop:negative-excess}, \ref{prop:positive} and \ref{prop:zero-excess}, a minimizer cannot have negative, positive or zero excess unless it is a ball centered at the origin. Hence any minimizer must be such a ball. Evaluating $\Lambda_p$ on $B(r)$ gives $\Lambda_p(B(r))=\omega_n^{-p/n}$ independently of $r$, which completes the proof.
\end{proof}

\section{A preliminary \texorpdfstring{$p$}{p}-barycenter}
\label{sec:preliminary}

For completeness, in this section we record an elementary fact that guarantees that, after a suitable translation, the coordinate functions can serve as test functions for the $p$-Steklov eigenvalue problem.

\begin{proposition}\label{prop:prelim-barycenter}
Let $\Omega\subset\mathbb{R}^n$ be a bounded convex domain and let $p>1$. For each $i=1,\dots,n$, there exists a unique $x_{i,0}\in\mathbb{R}$ such that
\[
\int_{\partial\Omega} |x_i - x_{i,0}|^{p-2}(x_i - x_{i,0})\ dS = 0.
\]
Set $x_0=(x_{1,0},\dots,x_{n,0})$. Then the translated domain $\Omega_0=\Omega-x_0$ satisfies
\[
\int_{\partial\Omega_0} |y_i|^{p-2}y_i\ dS = 0 \qquad\forall\ i=1,\dots,n.
\]
Consequently, the coordinate functions $y_i$ are admissible test functions in the variational characterization of the first nontrivial $p$-Steklov eigenvalue on $\Omega_0$.
\end{proposition}

\begin{proof}
Fix $i$ and define
\[
g_i(t)=\int_{\partial\Omega} |x_i-t|^{p-2}(x_i-t)\ dS,\qquad t\in\mathbb{R}.
\]
Note that the function $\varphi(s)=|s|^{p-2}s$ is continuous, strictly increasing, odd, and $\lim_{s\to\pm\infty}\varphi(s)=\pm\infty$. Since $x_i$ is bounded on the compact set $\partial\Omega$, $g_i$ is continuous. For $t$ sufficiently large, the inequality $x_i-t\leq-1$ holds, which implies that $\varphi(x_i-t)\leq-1$. Thus $g_i(t)\leq-\mathcal{H}^{n-1}(\partial\Omega)$, and letting $t\to+\infty$ gives $g_i(t) \to -\infty$. The same argument with $t \to -\infty$ yields $\lim_{t\to-\infty}g_i(t)=+\infty$. Since $\varphi$ is strictly increasing and $x_i-t$ is strictly decreasing in $t$, $g_i$ is strictly decreasing. The intermediate value theorem yields a unique $x_{i,0}\in\mathbb{R}$ with $g_i(x_{i,0})=0$. Set $x_0=(x_{1,0},\dots,x_{n,0})$.

Now let $\Omega_0=\Omega-x_0$ and write $y=x-x_0$. By the translation invariance of the boundary measure, we have
\[
\int_{\partial\Omega_0} |y_i|^{p-2}y_i\ dS
   = \int_{\partial\Omega} |x_i-x_{i,0}|^{p-2}(x_i-x_{i,0})\ dS = 0,
\]
which completes the proof.
\end{proof}

We present an example demonstrating that integral conditions $\int_{\partial \Omega} |x|^{p-2}x\  dS = 0$ and for each $i = 1, \dots, n$,  $\int_{\partial \Omega} |x_i|^{p-2}x_i\  dS = 0$ are not equivalent.
\begin{example}\label{ex:non-equivalence}
Let $n=2$, $p=3$. Let $r>0$. Define
\[
\Omega=\Omega_1\cup\Omega_2,\quad
\Omega_1=\{x_1^2+x_2^2<1,\;x_1\geq 0\},\quad
\Omega_2=(-r,0)\times(-1,1).
\]
The boundary $\partial\Omega$ consists of the right semicircle $\Gamma_1$, the horizontal segments $L_1=[-r,0]\times\{1\}$, $L_3=[-r,0]\times\{-1\}$,
and the vertical segment $L_2=\{-r\}\times[-1,1]$; see Figure~\ref{fig:example}.
\begin{figure}
 \begin{center}
  \includegraphics[width=0.6\textwidth]{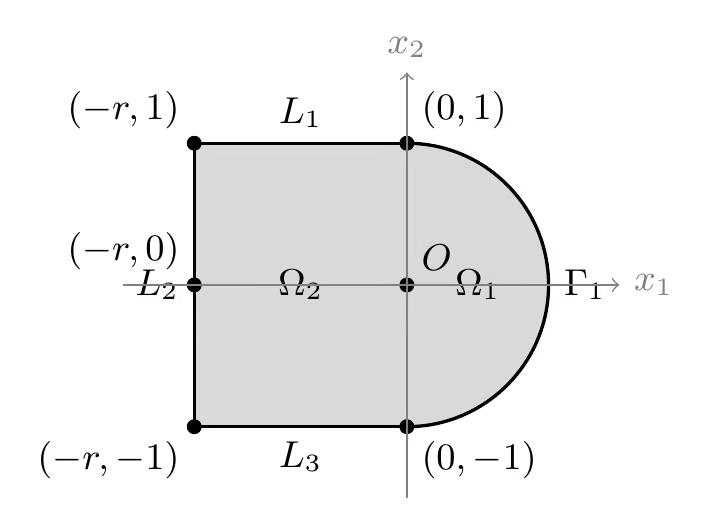}
  \caption{$\Omega$ is the union of a semicircle and a rectangle.}
  \label{fig:example}
 \end{center}
\end{figure}

On $\Gamma_1$ parametrized by $(x_1,x_2)=(\cos\theta, \sin\theta)$ with $\theta\in(-\pi/2,\pi/2)$, we have
\begin{align*}
\int_{\Gamma_1} |x|x\ dS &= \int_{-\pi/2}^{\pi/2}(\cos\theta,\sin\theta)\ d\theta = (2,0),\\
\int_{\Gamma_1} |x_1|x_1\ dS &= \int_{-\pi/2}^{\pi/2}\cos^2\theta\ d\theta = \frac{\pi}{2},\\
\int_{\Gamma_1} |x_2|x_2\ dS &= \int_{-\pi/2}^{\pi/2}|\sin\theta|\sin\theta\ d\theta = 0.
\end{align*}  
On $L_1=[-r,0]\times\{1\}$, parametrized by $x_1\in[-r,0]$ with $dS=dx_1$, we have
\begin{align*}
\int_{L_1}|x|x\ dS &=
\Bigl(\int_{-r}^0 x_1\sqrt{x_1^2+1}\ dx_1,\;
\int_{-r}^0\sqrt{x_1^2+1}\ dx_1\Bigr)\\[2pt]
&= \Bigl(\frac13-\frac13(r^2+1)^{3/2},\;
\frac{r}{2}\sqrt{r^2+1}-\frac12\log\bigl(\sqrt{r^2+1}-r\bigr)\Bigr),\\[4pt]
\int_{L_1}|x_1|x_1\ dS &= \int_{-r}^0 (-x_1)x_1\ dx_1 = -\frac{r^3}{3},\\[4pt]
\int_{L_1}|x_2|x_2\ dS &= \int_{-r}^0 1^2\ dx_1 = r.
\end{align*}
On $L_2=\{-r\}\times[-1,1]$, parametrized by $x_2\in[-1,1]$ with $dS=dx_2$, we have
\begin{align*}
\int_{L_2}|x|x\ dS &=
\Bigl(\int_{-1}^1 (-r)\sqrt{r^2+x_2^2}\ dx_2,\;
\int_{-1}^1x_2\sqrt{r^2+x_2^2}\ dx_2\Bigr)\\
&=\Bigl(-2r\int_0^1\sqrt{r^2+t^2}\ dt,\;0\Bigr)\\
&= \Bigl(-r\sqrt{r^2+1}-r^3\log\frac{1+\sqrt{r^2+1}}{r},\;0\Bigr),\\[4pt]
\int_{L_2}|x_1|x_1\ dS &= \int_{-1}^1 r(-r)\ dx_2 = -2r^2,\\[4pt]
\int_{L_2}|x_2|x_2\ dS &= \int_{-1}^1 |x_2|x_2\ dx_2 = 0.
\end{align*}
On $L_3=[-r,0]\times\{-1\}$, parametrized by $x_1\in[-r,0]$ with $dS=dx_1$, we have
\begin{align*}
\int_{L_3}|x|x\ dS &=
\Bigl(\int_{-r}^0 x_1\sqrt{x_1^2+1}\ dx_1,\;
\int_{-r}^0 (-1)\sqrt{x_1^2+1}\ dx_1\Bigr)\\
&= \Bigl(\frac13-\frac13(r^2+1)^{3/2},\;
-\frac{r}{2}\sqrt{r^2+1}+\frac12\log\bigl(\sqrt{r^2+1}-r\bigr)\Bigr),\\[4pt]
\int_{L_3}|x_1|x_1\ dS &= \int_{-r}^0 (-x_1)x_1\ dx_1 = -\frac{r^3}{3},\\[4pt]
\int_{L_3}|x_2|x_2\ dS &= \int_{-r}^0  (-1)\ dx_1 = -r.
\end{align*}
Hence,
\begin{align*}
\int_{\partial\Omega} |x|\ x\ dS &= \left(\frac{8}{3} - \frac{2}{3}(r^2+1)^{3/2} - r\sqrt{r^2+1} - r^3\log\frac{1+\sqrt{r^2+1}}{r},\ 0 \right),\\
\int_{\partial\Omega} |x_1|\ x_1\ dS &= \frac{\pi}{2} - 2r^2 - \frac{2}{3}r^3,\qquad 
\int_{\partial\Omega} |x_2|\ x_2\ dS = 0.
\end{align*}

Define $f(r)=\frac{2}{3}r^3+2r^2$ for $r>0$, with $f(0)=0$. Then $f'(r)=2r^2+4r>0$, so $f$ is strictly increasing on $[0,\infty)$. Since $f(3/4)=45/32<\pi/2$ and $f(r)\to\infty$ as $r\to\infty$, there exists a unique $r_1>3/4$ satisfying $f(r_1) = \pi/2$. For this value of $r_1$, we obtain
\[
\int_{\partial\Omega}|x_1|x_1\ dS = \frac{\pi}{2}-f(r_1)=0.
\]
Now consider the first component of the vector integral,
\[
g(r) = \frac{8}{3} - \frac{2}{3}(r^2+1)^{3/2} - r\sqrt{r^2+1} - r^3\log\frac{1+\sqrt{r^2+1}}{r}, \qquad r>0.
\]
Differentiating gives $g'(r)<0$ for all $r>0$, hence $g$ is strictly decreasing. Evaluating at $r=3/4$ gives
\[
g\left(\frac{3}{4}\right) = \frac{41}{96} - \frac{27}{64}\log 3 < 0.
\]
Since $r_1>3/4$ and $g$ is strictly decreasing, we obtain $g(r_1)<g(\tfrac{3}{4})<0$. Thus $\int_{\partial\Omega}|x|x\ dS \neq (0,0)$, whereas $\int_{\partial\Omega}|x_i|x_i\ dS = 0$ for $i=1,2$, showing that the two integral conditions are not equivalent.
\end{example}

\section{Proof of Theorem~\ref{thm:main-bound}}
\label{sec:proof-main}

We now give the proof of the upper bound for $\sigma_{1,p}$ with $p \in (1, n]$, which was stated as Theorem~\ref{thm:main-bound}.

\begin{proof}[Proof of Theorem~\ref{thm:main-bound}]

By Proposition~\ref{prop:prelim-barycenter}, there exists a point $x_0$ such that, after translating by $x_0$, the domain $\Omega_0 = \Omega - x_0$ satisfies
\[
\int_{\partial\Omega_0} |y_i|^{p-2} y_i\ dS = 0, \qquad \text{for every } i=1,\dots,n.
\]
In particular, the coordinate functions $y_i$ are admissible test functions in the variational characterization \eqref{eq:rayleigh} on $\Omega_0$. By the translation invariance of the $p$-Steklov spectrum, we have
\begin{equation}\label{eq:eigenvalue-trans-invar} 
\sigma_{1,p}(\Omega) = \sigma_{1,p}(\Omega_0).    
\end{equation}

For each $i$, using $u(y)=y_i$ in the Rayleigh quotient \eqref{eq:rayleigh} gives
\[
\sigma_{1,p}(\Omega_0) \int_{\partial\Omega_0} |y_i|^p\ dS \leq \int_{\Omega_0} |\nabla y_i|^p\ dy = V(\Omega_0).
\]
Summing over $i=1,\dots,n$ yields
\begin{equation}\label{eq:key-inequ-thm-proof}
\sigma_{1,p}(\Omega_0) \sum_{i=1}^n \int_{\partial\Omega_0} |y_i|^p\ dS\leq n V(\Omega_0).    
\end{equation}
 
When $p\in(1,2]$, the elementary inequality $|y|^p \leq \sum_{i=1}^n |y_i|^p$ holds pointwise, so that
\begin{equation}\label{eq:W_p-(1,2]}
W_p(\Omega_0) = \int_{\partial\Omega_0} |y|^p\ dS \leq \sum_{i=1}^n \int_{\partial\Omega_0} |y_i|^p\ dS.    
\end{equation}
When \(p\in(2,n]\), convexity of \(t\mapsto t^{p/2}\) on \([0,\infty)\) gives
\[
\left(\frac{1}{n} \sum_{i=1}^n y_i^2 \right)^{p/2} \leq \frac{1}{n} \sum_{i=1}^n (y_i^2)^{p/2}.
\]
Hence
\begin{equation}\label{eq:W_p-(2,n]}
W_p(\Omega_0) \leq n^{p/2 - 1}\sum_{i=1}^n \int_{\partial\Omega_0} |y_i|^p dS.    
\end{equation}
Define
\[
A(n,p):=
\begin{cases}
1, & 1<p\leq 2,\\
n^{p/2-1}, & 2<p\leq n.
\end{cases}
\]
Combining \eqref{eq:W_p-(1,2]} and \eqref{eq:W_p-(2,n]}, we have
\begin{equation}\label{eq:wp_bound}
 W_p(\Omega_0)\leq  A(n,p)  \sum_{i=1}^n \int_{\partial\Omega_0} |y_i|^p\ dS.
\end{equation}

Combining \eqref{eq:key-inequ-thm-proof} with \eqref{eq:wp_bound}, together with \eqref{eq:eigenvalue-trans-invar} and $V(\Omega_0)=V(\Omega)$, we obtain
\begin{equation}\label{eq:sigma_W_bound}
\sigma_{1,p}(\Omega_0) W_p(\Omega_0) \leq A(n,p)  n V(\Omega_0),
\qquad\text{i.e.,}\qquad
\sigma_{1,p}(\Omega) \leq A(n,p)\ \frac{n V(\Omega)}{W_p(\Omega_0)}.
\end{equation}
By applying Theorem~\ref{thm:Lambda-lower}, we have
\[
\frac{W_p(\Omega_0)}{P(\Omega_0)\ V(\Omega_0)^{p/n}} =\Lambda_p(\Omega_0)\geq \omega_n^{-p/n}.
\]
Since $P(\Omega_0)=P(\Omega)$ and $V(\Omega_0)=V(\Omega)$, 
\[
\frac{nV(\Omega)}{W_p(\Omega_0)} \leq  \frac{n V(\Omega)}{\omega_n^{-p/n} P(\Omega) V(\Omega)^{p/n}} = \frac{n\omega_n^{p/n}}{P(\Omega)} V(\Omega)^{1-\tfrac{p}{n}}.
\]
Applying the classical isoperimetric inequality $P(\Omega) \geq n\omega_n^{\frac{1}{n}} V(\Omega)^{\frac{n-1}{n}}$,
we obtain  
\[
\frac{nV(\Omega)}{W_p(\Omega_0)} \leq \frac{n\omega_n^{p/n}}{P(\Omega)}\left(\frac{P(\Omega)}{n\omega_n^{1/n}}\right)^{\frac{n-p}{n-1}} = \left(\frac{n\omega_n}{P(\Omega)}\right)^{\frac{p-1}{n-1}} = \frac{1}{r(\Omega^*)^{p-1}},
\]
where $\Omega^*$ is a ball with the same perimeter as $\Omega$. By \eqref{eq:sigma_W_bound}, we have 
\[
\sigma_{1,p}(\Omega) \leq A(n,p)\frac{1}{r(\Omega^*)^{p-1}}.
\]

If equality holds, then $\Lambda_p(\Omega_0) = \omega_n^{-p/n}$ and equality holds in \eqref{eq:wp_bound}. $\Lambda_p(\Omega_0) = \omega_n^{-p/n}$ implies that $\Omega_0$ is a ball centered at the origin. Hence, equality in \eqref{eq:wp_bound} holds if and only if $p = 2$. This completes the proof.
\end{proof}

\section{Proof of Theorem~\ref{thm:Wentzell-V}}
\label{sec:Wentzell}
In this section, we prove Theorem~\ref{thm:Wentzell-V}.

\begin{proof}[Proof of Theorem~\ref{thm:Wentzell-V}]
By Proposition~\ref{prop:prelim-barycenter}, there exists a point $x_0 \in \mathbb{R}^n$ such that $\Omega_0=\Omega-x_0$ satisfies
\begin{equation}\label{eq:ortho}
\int_{\partial\Omega_0}|y_i|^{p-2}y_i\ dS=0, \qquad \text{for every } i=1,\dots,n.
\end{equation}
Hence, each coordinate function $y_i$ is admissible in \eqref{eq:Wentzell-def}. Testing with $u=y_i$, summing over $i$ and using
$|\nabla y_i|^p=1$, $|\nabla_\tau y_i|^p=(1-\nu_i^2)^{p/2}$, we obtain
\begin{equation}\label{eq:sum}
\mu_{1,p}(\Omega_0,\beta)\sum_{i=1}^n\int_{\partial\Omega_0}|y_i|^p\ dS \leq nV(\Omega_0)+\beta\sum_{i=1}^n\int_{\partial\Omega_0}(1-\nu_i^2)^{p/2}\ dS.
\end{equation}

For the tangential term we use that $t\mapsto t^{p/2}$ is concave on $[0,1]$ when $p\leq2$. Then Jensen's inequality gives
\[
\frac{1}{n}\sum_{i=1}^n(1-\nu_i^2)^{p/2} \leq \left(\frac{1}{n}\sum_{i=1}^n(1-\nu_i^2)\right)^{p/2} = \left(\frac{n-1}{n}\right)^{p/2}.
\]
When $p \in (2, n]$, since the function $t\mapsto t^{p/2}$ is convex on $[0,1]$, we have
\[
 \sum_{i = 1}^n (1-\nu_i^2)^{p/2} \leq \sum_{i= 1}^n (1-\nu_i^2) = (n-1). 
\]
Set
\[
 B(n, p):= 
  \begin{cases}
    n\left(\frac{n-1}{n}\right)^{p/2}, \quad &1< p \leq 2, \\
    n-1, &2<p \leq n.
  \end{cases}
\]
We have 
\begin{equation}\label{eq:tang-est}
\sum_{i=1}^n\int_{\partial\Omega_0}(1-\nu_i^2)^{p/2}\ dS \leq B(n, p)P(\Omega_0).
\end{equation}

For the left‑hand side of \eqref{eq:sum}, the same distinction as in the proof of Theorem~\ref{thm:main-bound}
gives
\[
W_p(\Omega_0) \leq A(n,p) \sum_{i=1}^n \int_{\partial\Omega_0} |y_i|^p\ dS,
\]
where $A(n,p)=1$ for $1<p\leq 2$ and $A(n,p)=n^{p/2-1}$ for $2<p\leq n$. Inserting this and \eqref{eq:tang-est} into \eqref{eq:sum} and dividing by $W_p(\Omega_0)>0$, we obtain
\[
\mu_{1,p}(\Omega_0,\beta)\leq A(n,p)\frac{nV(\Omega_0)+\beta B(n, p)P(\Omega_0)}{W_p(\Omega_0)}.
\]

Translation invariance of $V,P$ and $\mu_{1,p}$ gives $V(\Omega_0)=V(\Omega)$, $P(\Omega_0)=P(\Omega)$ and $\mu_{1,p}(\Omega_0,\beta)=\mu_{1,p}(\Omega,\beta)$. By applying Theorem~\ref{thm:Lambda-lower}, we have
\[
W_p(\Omega_0)\geq\omega_n^{-p/n}P(\Omega)\ V(\Omega)^{p/n}.
\]
Substituting this into the estimate above, we obtain
\begin{equation}\label{eq:mu-final}
\mu_{1,p}(\Omega,\beta)\leq A(n,p)\left(n\omega_n^{p/n}\frac{V(\Omega)^{1-p/n}}{P(\Omega)} +\beta B(n, p)\omega_n^{p/n}V(\Omega)^{-p/n}\right).
\end{equation}
Then the classical isoperimetric inequality $P(\Omega)\geq n\omega_n^{1/n}V(\Omega)^{(n-1)/n}$ implies
\[
\frac{V(\Omega)^{1-p/n}}{P(\Omega)} \leq \frac{V(\Omega)^{1-p/n}}{n\omega_n^{1/n}V(\Omega)^{(n-1)/n}} =\frac{1}{n\omega_n^{1/n}}V(\Omega)^{-\frac{p-1}{n}}=\frac{1}{n\omega_n^{p/n}}\frac{1}{r(\Omega^*)^{p-1}},
\]
where $\Omega^*$ is a ball in $\mathbb{R}^n$ with the same volume as $\Omega$ centered at $0$. Hence,
\[
\mu_{1,p}(\Omega,\beta)\leq A(n,p)\left(\frac{1}{r(\Omega^*)^{p-1}}+\beta B(n, p)\frac{1}{r(\Omega^*)^p}\right).
\]
Equality case will be the same as in Theorem~\ref{thm:main-bound}. This completes the proof.
\end{proof}

\appendix\label{sec:appendix}
\section{Derivation of the first variation of \texorpdfstring{$\Lambda_p$}{Lambda	extunderscore p}}
In this appendix, We give a detailed derivation of equation \eqref{eq:Lambda-prime}. Let $\Omega$ be a bounded convex domain with $C^2$ boundary. For a smooth vector field $\theta\in C^\infty(\mathbb R^n,\mathbb R^n)$, consider the perturbed domains $\Omega_t:=(\mathrm{Id}+t\theta)\Omega$. We recall two classical shape differentiation formulas (see, e.g., \cite[Theorem~5.2.2]{HenrotPierre2005}, \cite[Section~2.11]{SZ92}). Although their precise statements involve Sobolev differentiability conditions on the pullbacks of the integrands, the functions used below ($f\equiv1$, $g\equiv1$ and $g(t,x)=|x|^p$) are smooth, so the necessary hypotheses are trivially satisfied. For a volume integrand $f$, set $J(t) = \int_{\Omega_t} f(t,x)\ dx$. The shape derivative of $J$ at $t=0$ is defined by the limit
\[
J'(0) = \frac{d}{dt}\bigg|_{t=0}\int_{\Omega_t} f(t,x)\ dx
      = \lim_{h\to 0} \frac{J(h)-J(0)}{h}.
\]
Under suitable regularity assumptions on $f$ and $\Omega$, this limit exists and equals the integral
\begin{equation}\label{thm:deriv-1}
\frac{d}{dt}\bigg|_{t=0}\int_{\Omega_t} f(t,x)\ dx
= \int_\Omega \left(\partial_t f(0,x) + \operatorname{div}(f(0,x)\ \theta)\right) dx.
\end{equation}
Similarly, for a boundary integrand $g$,
\begin{equation}\label{thm:deriv-2}
 \begin{split}
  &\mathrel{\phantom{=}}\frac{d}{dt}\bigg|_{t=0}\int_{\partial\Omega_t} g(t,x)\ dS \\
  &= \int_{\partial\Omega} \left(\partial_t g(0,x) + \langle\theta,\nu\rangle \left(\nabla g(0,x)\cdot\nu + (n-1)H\ g(0,x)\right)\right) dS.
 \end{split}
\end{equation}
Applying \eqref{thm:deriv-1} with $f\equiv1$ and \eqref{thm:deriv-2} with $g\equiv1$ immediately gives
\begin{equation}\label{eq:Vprime}
\frac{d}{dt}\bigg|_{t=0} V(\Omega_t) = \int_{\partial\Omega} \langle\theta,\nu\rangle\ dS
\end{equation}
and
\begin{equation}\label{eq:Pprime}
\frac{d}{dt}\bigg|_{t=0} P(\Omega_t) = (n-1)\int_{\partial\Omega} H\ \langle\theta,\nu\rangle\ dS. 
\end{equation}
For $W_p(\Omega_t)$ we use \eqref{thm:deriv-2} with $g(t,x)=|x|^p$. Since $\partial_t g=0$, we have
\begin{align}\label{eq:Wprime}
\frac{d}{dt}\bigg|_{t=0} W_p(\Omega_t)
&= \int_{\partial\Omega} \langle\theta,\nu\rangle
   \left( \nabla(|x|^p)\cdot\nu + (n-1)H|x|^p \right) dS \notag\\
&= \int_{\partial\Omega} \langle\theta,\nu\rangle
   \left( p|x|^{p-2}\langle x,\nu\rangle + (n-1)H|x|^p \right) dS.
\end{align}

Now differentiate $\Lambda_p(\Omega_t) = W_p(\Omega_t)/\left(P(\Omega_t)V(\Omega_t)^{p/n}\right)$ at $t=0$. We have
\begin{align*}
\frac{d}{dt}\bigg|_{t=0} \Lambda_p(\Omega_t) &= \frac{1}{P(\Omega) V(\Omega)^{p/n}}\frac{d}{dt}\bigg|_{t=0} W_p(\Omega_t)  \\
&\mathrel{\phantom{=}}- \frac{W_p(\Omega)}{P(\Omega)^2V(\Omega)^{p/n}}\frac{d}{dt}\bigg|_{t=0} P(\Omega_t)- \frac{p}{n}\frac{W_p(\Omega)}{P(\Omega) V(\Omega)^{p/n+1}}\frac{d}{dt}\bigg|_{t=0} V(\Omega_t).
\end{align*}
Inserting the expressions \eqref{eq:Vprime}, \eqref{eq:Pprime} and \eqref{eq:Wprime} and collecting the terms gives
\begin{equation*}
\begin{split}
\frac{d}{dt}\bigg|_{t=0} \Lambda_p(\Omega_t) &= \frac{1}{P(\Omega)V(\Omega)^{p/n}}\int_{\partial\Omega} \langle\theta,\nu\rangle \left(p|x|^{p-2}\langle x,\nu\rangle + (n-1)H|x|^p\right) dS \\
&\mathrel{\phantom{=}}- \frac{W_p(\Omega)}{P(\Omega)^2V(\Omega)^{p/n}}\int_{\partial\Omega} (n-1)H\langle\theta,\nu\rangle\ dS  \\
&\mathrel{\phantom{=}}- \frac{p}{n}\frac{W_p(\Omega)}{P(\Omega) V(\Omega)^{p/n+1}}\int_{\partial\Omega} \langle\theta,\nu\rangle\ dS \\
&= \frac{1}{P(\Omega)V(\Omega)^{p/n}}\int_{\partial\Omega} p\langle\theta,\nu\rangle\left(|x|^{p-2}\langle x,\nu\rangle - \frac{W_p(\Omega)}{nV(\Omega)}\right) dS \\
&\mathrel{\phantom{=}} + \frac{1}{P(\Omega)V(\Omega)^{p/n}}\int_{\partial\Omega} (n-1)H\langle\theta,\nu\rangle\left(|x|^p - \frac{W_p(\Omega)}{P(\Omega)}\right) dS.
\end{split}
\end{equation*}
This is precisely equation \eqref{eq:Lambda-prime}.

\section*{Acknowledgments}
The first author is supported by NSFC (Nos.~12101125 and 12371052) and the Fujian Alliance of Mathematics (No.~2024SXLMMS01). The second author is partially supported by the National Key R\&D Program of China (No.~2025YFA1017500).

\bibliographystyle{plain}
\bibliography{reference}
\end{document}